\documentclass[12pt]{article}
\RequirePackage{amsmath,amsthm}
\RequirePackage[colorlinks,citecolor=blue,urlcolor=blue,pagebackref=true]{hyperref}
\RequirePackage[T1]{fontenc}
\usepackage[british,french]{babel}
\usepackage{amssymb,math,locenglish}
\usepackage{scrtime}
\usepackage{fourier}
\usepackage{stmaryrd}

\newcommand{\eps}{\varepsilon}
\newtheorem{cond}{Condition}
\newcommand{\1}{\id}
\renewcommand{\cut}[1]{}

\usepackage[usenames]{color}
\usepackage{tikz}

\newcommand{\LastUpdate}{\eurtoday\ at \thistime}
\usepackage{amsmath,amssymb,amsthm}
\usepackage{enumerate}
\usepackage[british]{babel}

\cut{
\theoremstyle{plain}
\newtheorem{theo}{Theorem}

\theoremstyle{definition}

\theoremstyle{remark}

\allowdisplaybreaks

\def\P{{\mathbf P}}
\def\1{{\mathbf 1}}

\def\eps{\varepsilon}
\let\phi=\varphi

}

\allowdisplaybreaks

\begin{document}

\begin{center}
\baselineskip=5mm
\textbf
{\Large
{
{Dynamical systems \\
with heavy-tailed random parameters}}}
 
\vskip1cm
\parbox[t]{14cm}{
Vladimir {\sc   Belitsky}$^{a}$\\
 Mikhail {\sc  Menshikov}$^{b}$\\
  Dimitri {\sc Petritis}$^c$ \\ Marina {\sc  Vachkovskaia}$^{d}$\\
\vskip5mm   
{\scriptsize  
\baselineskip=5mm
a. Instituto de Matem{\'a}tica e Estat{\'\i}stica,
Universidade de S{\~a}o Paulo, rua do Mat{\~a}o 1010, CEP
05508--090, S{\~a}o Paulo, SP, Brazil,
belitsky@ime.usp.br

b. Department of Mathematical Sciences,
University of Durham,
South Road,
Durham DH1 3LE,United Kingdom\\
Mikhail.Menshikov@durham.ac.uk

c. Institut de Recherche
Math\'ematique, Universit\'e de Rennes I,
Campus de Beaulieu, 
35042 Rennes, France, dimitri.petritis@univ-rennes1.fr

d. Department of Statistics,
 Institute of Mathematics, Statistics and Scientific Computation,
University of Campinas--UNICAMP, Rua Sergio Buarque de Holanda, 651
  Campinas, SP, CEP 13083-859, Brazil\\
   marinav@ime.unicamp.br

}}
\vskip10mm
\LastUpdate
\end{center}
\vskip5mm

\selectlanguage{british}

\begin{abstract}
{\scriptsize
Motivated by the study of the time evolution of random dynamical systems arising in a vast variety of domains
--- ranging from physics to ecology ---, we establish conditions for the occurrence of a non-trivial asymptotic behaviour for these systems in the absence of an ellipticity condition.  More precisely, we classify these systems according to their type and --- in the recurrent case --- provide with sharp conditions quantifying the nature of recurrence by establishing  which moments of passage times exist and which do not exist.  The problem is tackled by  mapping the random dynamical systems into Markov chains on $\BbR$ with heavy-tailed innovation and then using powerful methods stemming from  Lyapunov functions to map the resulting Markov chains into positive semi-martingales. 
\\[0.8cm]
}
\end{abstract}

\section{Introduction}
\subsection{Motivation}
The  theory of dynamical systems aims at describing the time evolution of a rich variety of systems: physical, chemical, ecological, biological, social, economical,  financial, computational etc.\ by sampling the continuous-time evolution at discrete time epochs. The evolution during a unit of  time is encoded into a nonlinear transformation $T$ from some metric space $\BbX$ into itself (usually $\BbX\subseteq \BbR$ (or $\BbR^d$), equipped with its Borel $\sigma$-algebra). Thus, generically, a dynamical system is described by a sequence $(X_n)_{n\in\BbN}$ of  state variables $X_n\in\BbX$ defined by the iteration
$X_{n+1}=T(X_n)$, for  $n\geq 0$.  

The \textit{nonlinearity} of $T$ induces a chaotic behaviour on the trajectory  $(X_n)_{n\in\BbN}$ and although the above evolution is purely deterministic, one can prove, under some conditions on $T$, 
ergodic theorems,  central limit theorems etc.\ (see for instance \cite{LasotaMackey1994}).

\cut{Even the evolution can be described by a \textit{deterministic Markovian  kernel}   $P(x,B)=\BbP(X_{n+1}\in B| X_n=x)= \delta_{T(x)}(B)$, for all Borel subsets $B\in\cB(\BbX)$  \cite{LasotaMackey1999} (this kernel is the conditional spectral measure of $T$). 
}

In realistic models, the transformation $T$ is not universal but depends on a certain number of external parameters, modelling the effect of the environment. Since the dynamics of the environment is complicated and the  control on it is poor, 
it is very natural to assume that the control parameters are random \cite{AthreyaMajumdar2003}. Let  $(\BbA, \cA)$ be a measurable space and suppose that $(A_n)_{n\in\BbN}$ are a sequence of $\BbA$-valued independent identically distributed random variables, defined on some abstract probability space, having common law $\nu $ and $(T_a)_{a\in\BbA}$ a family of transformations $T_a:\BbX\rightarrow \BbX$ indexed by the set $\BbA$. Then, a random dynamical system driven by the sequence $(A_n)_{n\in\BbN}$ reads
$X_{n+1}=T_{A_{n+1}}(X_n)$, for $n\geq 0$. 
\cut{The evolution kernel of the dynamical system becomes subsequently a genuine 
stochastic kernel 
\[P(x,B)=\BbP(X_{n+1}\in B| X_n =x;\cA_n)= \int_\BbA\BbP(A_{n+1}\in da)\delta_{T_{a}(x)}(B)= \int_\BbA \nu (da)\delta_{T_{a}(x)}(B),\] where $\cA_n=\sigma(A_1, \ldots, A_n)$ and $(X_n)$ a genuine Markov chain. (We assume in this paper that there exists a probability space $(\Omega, \cF, \BbP)$ sufficiently large to hold all the random variables of the problem, i.e.\ $(X_n)$ and $(A_n)$). }
Our work is motivated by  models stemming from a subclass of multiplicative transformations that have been thoroughly studied in the literature;  namely, we assume that  $\BbA=\BbR^+$, $\BbX=\BbR^+$, and there exists a single continuous (on $[0,\infty[$) and differentiable (on $]0,\infty[$) transformation $T:\BbR^+\rightarrow \BbR^+$,  such that  the whole family
is defined through $T_a(\cdot) =T(a \ \cdot)$ for all $a\in\BbA$.  This class of models have been studied in \cite{Athreya2004, AthreyaDai2000, AthreyaDai2002} under the condition of uniform ellipticity, reading $T'(0+)=C>0$. 
When the uniform ellipticity condition is not satisfied, the situation is considerably harder even for deterministic dynamical systems 
\cite{HasselblattKatok2002}. 

The novelty of our paper lies in the fact that we treat a class of models where \textit{uniform ellipticity fails} (i.e.\ allowing $T'(0+)=0$). 
We  are able to answer the question whether the 
 process visits a small region near the origin in finite time; this result constitutes  the main step towards establishing that  the invariant measure of the stochastic dynamical system
 $(X_n)_n$  generated by the recursive 
relation  $X_{n+1}=T(A_{n+1} X_n)$ is  the Dirac mass, $\delta_0$, concentrated at 0.

\subsection{Description of the model}
Suppose that there exists a universal mapping $f:\BbR^+\rightarrow \BbR^+$ --- verifying certain conditions that will be precised later --- allowing to define the whole family of transformations through $T_a(x)=axf(ax)$, for $a\in\BbA$ and $x\in\BbR^+$.  We arrive thus at the following random dynamical 
system 
$X_{n+1}=A_{n+1} X_n f(A_{n+1}X_n)$,
where $(A_n)_{n\geq 1}$ are a sequence of independent and identically distributed  $\BbR^+$-valued random variables with law $\nu$. Not to complicate unnecessarily the model, we assume that $\nu$ has always a density, with respect to
either  the Lebesgue measure on the non-negative axis or the counting measure of some infinitely denumerable unbounded subset of $\BbR^+$.  We address the question about  the asymptotic behaviour of $X_n$, as $n\to \infty$. The situtation $\lim_{n\rightarrow \infty}X_n=0$ has a special significance since can be interpreted  as the extinction of certain natural resources, or the bankruptcy of certain financial assets, etc. The dual situation of 
$\lim_{n\rightarrow \infty}X_n=\infty$ can also be interpreted as the proliferation of certain species, or the creation of 
instabilities due to the formation of speculative bubbles, etc.\ (see \cite{BhattacharyaMajumdar2007} for instance).

Since the previous Markov chain is multiplicative, it is natural to work at logarithmic scale and consider the additive version of the dynamical system $\xi_{n+1}= \tau_{\alpha_{n+1}} (\xi_n)$; here $\xi_n=\ln X_n$, $\alpha_{n+1}=\ln A_{n+1}$,  $\tau_\alpha(\xi)= \xi+\alpha+ \psi(\xi+\alpha)$, with   $\psi(z)=\ln f(e^z)$, for $z\in\BbR$. Therefore, the Markov chain becomes now an $\BbR$-valued one reading
$\xi_{n+1}= \xi_n+\alpha_{n+1}+\psi(\xi_n+\alpha_{n+1}).$
Obviously, 
$\xi_n\to +\infty \ \textrm{a.s.}\ \Leftrightarrow X_n\to+\infty \ \textrm{a.s.}$
and 
$\xi_n\to -\infty \ \textrm{a.s.}\ \Leftrightarrow X_n\to0 \ \textrm{a.s.}$

An important  class of non-uniformly elliptic random dynamical systems are those $(X_n)$ that --- when considered at logarithmic scale as above --- have 
$\psi(t)=\pm |t|^\gamma$, for $0<\gamma<1$ and $t\in \BbR^+$.
 Now using the elementary inequalities (see \cite[\S19, p.\ 28]{HardyLittlewoodPolya1988}, for instance)
$a^\gamma-|b|^\gamma \le (a+b)^\gamma \le a^\gamma+|b|^\gamma,$
it turns out that the dynamical system reads 
$\xi_{n+1}=\xi_n+\alpha_{n+1}\pm|\xi_n+\alpha_{n+1}|^\gamma= \xi_n+\alpha_{n+1}\pm|\xi_n|^\gamma+ \cO(\alpha_{n+1}^\gamma).$
Now, for $\gamma\in ]0,1[$, the term $\cO(\alpha_{n+1}^\gamma)$ in the above expression turns out to be subdominant. 

For the aforementioned  reasons, 
we  study in this paper  the  Markov chains on $\BbX=\BbR^+$ defined by   one of the following recursions
\begin{align*}
\zeta_{n+1}&= (\zeta_n+\alpha_{n+1} -\zeta_n^\gamma)^+, \ \textrm{ or}\\
 \zeta_{n+1}&= (\zeta_n+\alpha_{n+1} +\zeta_n^\gamma)^+,
 \end{align*}
with $\gamma\in ]0,1[$ and $\zeta_0=x$ a.s.;  here $z^+=\max(0,z)$ and $x\in\BbX$.
The sequence $(\alpha_n)_{n\geq 1}$ are a family of independent $\BbR$-valued random variables having common distribution. This distribution can be supposed discrete or continuous but will always be assumed having one- or two-sided heavy tails. The heaviness of the tails is quantified by the order of the fractional  moments failing to exist.
 
\subsection{Main results}
In all statements below, we make the  
\begin{glas}
The sequence $(\alpha_n)_{n\in\BbN}$ are independent and identically distributed real random variables. 
The common law is denoted by $\mu$ and is supposed to be $\mu\ll \lambda$ where $\lambda$ is a reference measure on $\BbR$;
we denote by $m=\frac{d\mu}{d\lambda}$ the corresponding density. Additionally, $\mu$ is supposed to be heavy-tailed (preventing thus integrability of the random variables $\alpha_n$).
\end{glas}

Let $(\zeta_n)_{n\in\BbN}$ be a Markov chain on a measurable space $(\BbX, \cX)$; denote, as usual, by $\BbP_x$ the probability on the trajectory space conditioned to $\zeta_0=x$ and, for $A\in\cX$, define $\tau_A=\inf\{n\geq 1:\zeta_n\in A\}$. Our paper is devoted in establishing conditions under which the time $\tau_A$ is finite (a.s.)\ or infinite (with strictly positive probability) and in case it is a.s.\ finite which of its moments exist. These results constitute the first step  toward establishing more general results on the Markov chain like recurrence or transience, positive recurrence and existence of invariant probability, etc. However, the latter need more detailed conditions on the communication structure of states of the chain like  $\phi$-accessibility, $\phi$-recurrence, maximal irreducibility measures and so on (see \cite{Orey1971,Nummelin1984,MeynTweedie1993} for instance). All those questions are important but introduce some technicalities that blur the  picture that we wish to reveal here, namely that questions on $\tau_A$ can be answered with extreme parsimony on the  hypotheses imposed on the Markov chain,  by using Lyapunov functions. As a matter of fact, the only communication property imposed on the Markov chain is mere accessibility
whose definition is recalled here for the sake of completeness.
 \begin{defi}
 \label{def:accessibility}
 Let $(Z_n)$ be a  Markov chain on $(\BbX,\cX)$ with stochastic kernel $P$ and $A\in\cX$.  Denote by $\BbP$ the probability on its trajectory space induced by $P$ and by $\BbP_x$ the law of trajectories conditioned on  $\{Z_0=x\}$. We say that $A$ is \textbf{accessible} from $x\not \in A$, if $\BbP_x(\tau_A<\infty)>0$.
 \end{defi}

\begin{theo}
\label{thm:precise-right}
Let $(\zeta_{n+1})$ be the Markov chain defined by the recursion 
\[\zeta_{n+1}= \zeta_n -\zeta_n^\gamma +\alpha_{n+1}, \ n\geq 0,\]
where $0<\gamma<1$ and the random variables $(\alpha_n)$ have a common law $\mu$ supported by $\BbR_+$, satisfying the condition $\mu([0,1])>0$ and whose density with respect to the Lebesgue measure, for large $y>0$, reads $m(y) =\id_{\BbR^+}(y) c_y y^{-1-\theta}$, with $\theta\in ]0,1[$. 
Let $a>1$ and denote by $A:=A_a=[0,a]$. Then $A$ is accessible from any point $x>a$. Additionally, the following hold.
\begin{enumerate}
\item Assume that there exist constants $0<b_1< b_2<\infty$ such that  $b_1\leq c_y \leq b_2$ for all $y\in\BbX$.
\begin{enumerate}
\item If $\theta>1-\gamma$  then $\BbP_x(\tau<\infty)=1$.
Additionally, 
\begin{itemize}
\item if $q<\frac{\theta}{1-\gamma}$ then $ \BbE_x(\tau_A^q)<\infty$, and
\item if $q\geq\frac{\theta}{1-\gamma}$ then $\BbE_x(\tau_A^q)=\infty$.
\end{itemize}
\item If $\theta<1-\gamma$  then $\BbP_x(\tau_A<\infty)<1$.
\end{enumerate}
\item Assume further that $\lim_{y\rightarrow\infty} c_y=c>0$ and $\theta=1-\gamma$.
\begin{enumerate}
\item If  $c\pi\csc(\pi\theta) <\theta$ then $\BbP_x(\tau_A<\infty)=1$. 
Denote 
$K_{\delta,\theta}=\frac{\Gamma(1-\theta)\Gamma(\theta-\delta)}{\theta\Gamma(1-\delta)}$; then there exists a unique $\delta_0\in ]0, \theta[$ such that $cK_{\delta_0,\theta}=1$. Additionally, 
\begin{itemize}
\item if $q <\frac{\delta_0}{1-\gamma}$  then $\BbE_x(\tau_A^q)<\infty$, and
\item  if $q >\frac{\delta_0}{1-\gamma}$ then  $\BbE_x(\tau_A^q)=\infty$.
\end{itemize}
\item If $c\pi\csc(\pi\theta) >\theta$ then  $\BbP_x(\tau_A<\infty)<1$.

\end{enumerate}
\end{enumerate}
\end{theo}

\begin{theo}
\label{thm:precise-left}
Let $(\zeta_{n+1})$ be the Markov chain defined by the recursion 
\[\zeta_{n+1}= (\zeta_n +\zeta_n^\gamma -\alpha_{n+1})^+ \]
for $n\geq 0$,
where $0<\gamma<1$ and the common law of the  random variables  $(\alpha_n)$ is supported by $\BbR_+$ and  has density $m$ with respect to the Lebesgue measure verifying $m(y)=\id_{\BbR+}(y)c_y{y^{-1-\theta}}$ for large $y>0$, with $\theta\in ]0,1[$ . 
Assume further that $\lim_{y\rightarrow\infty} c_y=c>0$. Then the state $0$ is accessible and
\begin{enumerate}
\item 
If  $\theta< 1-\gamma$  then $\BbE_x(\tau_0^q)<\infty$, for all $q>0$. 
\item If  $\theta= 1-\gamma$  then $\BbP_x(\tau_0<\infty)=1$. Denote\footnote{It is recalled that the transcendental function $\Gamma$,  defined by $\Gamma(z):=\int_0^\infty \exp(-t) t^{z-1}dt$ for $\Re z>0$, can  be analytically continued on $\BbC\setminus\{0, -1, -2, -3,\ldots\}$; its analytic continuation can be expressed by $\Gamma(z)=\int_0^\infty [\exp(-t) -\sum_{m=0}^n \frac{(-t)^m}{m!}] t^{z-1}dt$ for 
$-(n+1) <\Re z <-n$ and $n\in\BbN$ (see \cite[\S1.1 (9), p.\ 2]{ErdelyiMagnusOberhettingerTricomi1981a} for instance).}
$L_{\delta,\theta}= \frac{\Gamma(1+\delta)\Gamma(-\theta)}{\Gamma(1-\theta+\delta)}$; then there exists a unique 
 $\delta_0\in ]0, \infty[$ such that $cL_{\delta_0, \theta}+ \delta_0=0$.
\begin{itemize}
\item If $q<\frac{\delta_0}{\theta}$ then $\BbE_x(\tau_0^q)<\infty$, and 
\item if $q>\frac{\delta_0}{\theta}$ then  $\BbE_x(\tau_0^q)=\infty$.
\end{itemize}
\item If $\theta>1-\gamma$  then $\BbP_x(\tau_0<\infty)<1$.
\end{enumerate}
\end{theo}

\begin{remn}
If $b_1\leq c_y \leq b_2$ but $c_y\not\rightarrow c$ then the conclusions established in the cases of  strict inequalities $\theta<1-\gamma$ or $\theta>1-\gamma$ remain valid. Nevertheless, we are unable to treat the critical case $\theta=1-\gamma$.  
\end{remn}

\begin{remn}
In both the above theorems,  the boundedness or existence of limit conditions on $(c_y)$ imply that the tails have power decay, i.e.\ there exists $C$ such  that the tail estimate $\BbP(\alpha>y)\geq \frac{C}{y^\theta}$ holds. Nevertheless, the control we impose is much sharper because we wish to treat the critical case. If we are not interested in the critical case, the control on $(c_y)$ can be considerably weakened by assuming only the tail estimate. Results established with such weakened control on the tails are given in theorems \ref{thm:systematic-left} and \ref{thm:systematic-right} below. 
\end{remn}

\begin{theo}
\label{thm:systematic-left}
 Let $(\zeta_n)$ be the Markov chain defined by the recursive relation 
\begin{equation}
 \label{eq_zeta1b}
\zeta_{n+1}=(\zeta_n-\zeta_n^\gamma +\alpha_{n+1})^+, \ n\geq 0,
\end{equation}
where $0<\gamma<1$ and the random variables $(\alpha_n)$ have common law with support extending to both negative and positive parts of the real axis. Let $a>1$ and denote by $A:=A_a=[0,a]$. Then $A$ is accessible and the following statements hold.
\begin{enumerate}
\item Suppose that there exist a positive constant
$C$ and a parameter $\theta\in \ ]0,1[$ such that $\P(\alpha_1>y)\le Cy^{-\theta}$. 
If  $\theta>1-\gamma$, then   $\forall q<\frac{\theta}{1-\gamma}$, $\BbE_x(\tau_A^q)<\infty$.
 \item Suppose that  there exist a positive constants $C,C'$ and parameters $\theta, \theta'$ with 
 $0<\theta<\theta'<1$ such that  $\P(\alpha_1>y)\ge C'y^{-\theta}$ and
 $\P(\alpha_1< -y)\le Cy^{-\theta'}$ (the right tails are heavier than the left ones). If   $\theta<1-\gamma$, then 
 $\BbP_x(\tau_A<\infty)<1$.
\end{enumerate}
\end{theo}

\begin{theo}
\label{thm:systematic-right}
Assume that the Markov chain $(\zeta_n)$ is defined by the recursive relation 
\begin{equation*}
\zeta_{n+1}=(\zeta_n+\zeta_n^\gamma+\alpha_{n+1})^+, \ n\geq 0,
\end{equation*}
where $0<\gamma<1$ and the random variables $(\alpha_n)$ have common law with support extending to both negative and positive parts of the real axis. Let $a>1$ and suppose that the set  $A:=A_a=[0,a]$ is accessible.
\begin{enumerate}
\item Suppose there exist a positive constant $C$ and a parameter $\theta$ with  $0<\theta<1$, such that  $\P(\alpha_1<-y)\leq Cy^{-\theta}$. 
If   $\theta>1-\gamma$, then $\BbP_x(\tau_A<\infty)<1$.
 \item Suppose there exist positive constant $C, C'$ and parameters $\theta$ and $\theta'$, with $0<\theta<\theta'<1$, such that 
  $\P(\alpha_1>y)\leq C'y^{-\theta'}$ and $\P(\alpha_1<-y)\geq Cy^{-\theta}$. If 
 $\theta<1-\gamma$ then the state $0$ is recurrent and  $\forall q<1$, $\BbE_x(\tau_A^q)<\infty$.
\end{enumerate}
\end{theo}

\section{Proofs}

\subsection{Results from the constructive theory of Markov chains}

The Markov chains we consider evolve on the set $\BbX=\BbR_+$. Our
 proofs rely on the possibility of constructing  measurable functions $g:\BbX\rightarrow\BbR^+$ (with some special properties regarding their asymptotic behaviour) that are superharmonic with respect to the discrete Laplacian operator $D=P-I$; consequently,  the image of the Markov chain under $g$ becomes a supermartingale outside some specific sets.
 For the convenience of the reader, we state here the principal theorems from the constructive theory, developed in \cite{FayolleMalyshevMenshikov1995} and in \cite{AspandiiarovIasnogorodskiMenshikov1996}, rephrased and adapted to the needs and notation of the present paper. We shall use repeatedly these theorems in the sequel.
 
 In the sequel 
 $(Z_n)$ denotes a  Markov chain on $\BbX$, having stochastic kernel $P$. We denote by
 \[\Dom_+(P):\{f:\BbX\to \BbR^+: f \textrm{ measurable  s.t.\ }\forall x\in\BbX,  \int_\BbX P(x,dy) f(y)<\infty\}.\] 
 We denote by $D=P-I$ the Markov operator whose action  $\Dom_+(P)\ni g \mapsto Dg$ reads 
 \[Dg(x)=\int_\BbX P(x,dy) g(y)- g(x)=  \BbE(g(Z_{n+1})-g(Z_n)|Z_n=x).\] 
  Notice that when $g$ is $P$-superharmonic, then $(g(Z_n))$ is a positive supermartingale.

 \begin{theo}
 [Fayolle, Malyshev, Menshikov \protect{\cite[Theorems 2.2.1 and 2.2.2]{FayolleMalyshevMenshikov1995}}]
 \label{thm:FMM95}
 Let $(Z_n)$ be a  Markov chain on $\BbX$ with kernel $P$ and  for $a\geq 0$, denote by $A:=A_a=[0,a]$. 
 \begin{enumerate}
 \item If there exist a pair $(f,x_0)$, where $x_0>0$ and $f\in\Dom_+(P)$ such that $\lim_{x\rightarrow \infty} f(x)=\infty$, $Df(x)\leq 0$ for all $x\geq x_0$, and $A:=A_{x_0}$ is accessible, then $\BbP_{x_0} (\tau_A<\infty)=1$.
 \item If there exist a pair $(f,A)$, where $A$ is a subset of $\BbX$ and $f\in\Dom_+(P)$ such that
 \begin{enumerate}
 \item $Df(x)\leq 0$ for $x\not \in A$, and
 \item there exists $y\in A^c: f(y) <\inf_{x\in A} f(x)$, 
 \end{enumerate}
 then $\BbP_{x_0} (\tau_A<\infty)<1$.
 \end{enumerate}
 \end{theo}

Let $f:\BbX\to\BbR_+$ and $a>0$. We denote $S_a(f)=\{x\in\BbX: f(x)\leq a\}$, the \textbf{sublevel set} of $f$.  We say that the function \textbf{tends to infinity}, $f\to\infty$, if $\forall n\in\BbN, \card S_n(f)<\infty$.

\begin{theo}
[Aspandiiarov, Iasnogorodski, Menshikov 
\protect{\cite[Theorems 1 and 2]{AspandiiarovIasnogorodskiMenshikov1996}}]
\label{thm:AIM96}
Let $(Z_n)$ be a  Markov chain on $\BbX$ with kernel $P$ and $f\in\Dom_+(P)$ such that $\lim_{x\to\infty} f(x)=\infty$.
\begin{enumerate}
\item If there exist strictly positive constants $a, p, c$ such that the set $A:=S_a(f)$ is accessible, $f^p\in\Dom_+(P)$, and $Df^p(x)\leq -cf^{p-2}(x)$ on $A^c$, then $\BbE_x(\tau_A^q) <\infty$ for all $q<p/2$.
\item It there exist $g\in\Dom_+(P)$ and 
\begin{enumerate}
\item a constant $b>0$ such that $f\leq bg$,
\item constants $a, c_1>0$ such that $Dg(x)\geq -c_1$ on $\{g>a\}$,
\item constants $c_2>0$ and $r>1$ such that $g^r\in\Dom_+(P)$ and $Dg^r(x) \leq c_2 g^{r-1}(x)$ on $\{g>a\}$,
\item a constant $p>0$ such that $f^p\in\Dom_+(P)$ and $Df^p(x) \geq 0$ on $\{f>ab\}$,
\end{enumerate}
then $\BbE_x(\tau^q_{S_{ab}(f)})=\infty$ for all $q>p$.
\end{enumerate}
\end{theo}

\begin{nota}
For $h:\BbR^+\rightarrow\BbR^+$, $\rho\in\BbR$,  we write $h(x)\asymp x^\rho$, if $\lim_{x\rightarrow \infty}h(x)x^{-\rho}=1$ and $h(x) \Yleft x^\rho$, if there exist a function $h_1$ such that $h(x)\leq h_1(x)$ and $h_1(x)\asymp
x^\rho$.
\end{nota}

\subsection{Proof of the theorems \ref{thm:precise-right} and \ref{thm:precise-left}}
The main theorems are stated under the condition that the reference measure  $\lambda$ is the Lebesgue measure on $\BbR$ (or on $\BbR_+$). To simplify notation, we write $\lambda(dy)=dy$ for Lebesgue measure. The case of 
$\mu$ having a density with respect to the counting measure on $\BbZ$ requires a small technical additional step as will be explained in the remark \ref{rem:sum-integral} below.

In the sequel, we shall use a Lyapunov function, $g$,  depending on a parameter $\delta\ne0$, reading 
\[g(x)=\left\{\begin{array}{ll}
x^\delta, & x\geq 1\\
1, &x<1
\end{array}\right. \ (\textrm{if }\ \delta<0)
\]
 and 
 \[g(x)=x^\delta \ (\textrm{if }\ \delta>0). 
\]
in general the choice $\delta>0$ is made to prove recurrence and $\delta<0$ to prove transience. The range of values of $\delta$ will be determined from the specific context as explained below.
\begin{lemm}
\label{lem:Dg}
Let $(\zeta_n)$ be the Markov chain of the theorem \ref{thm:precise-right} and
suppose that $x$ is very large.  
For arbitrary $y_0\geq 1$ and $\delta<\theta$, 
\[Dg(x)\Yleft (x-x^\gamma)^\delta \biggl[\int_{[y_0,\infty[}\left(\Big(1+ \frac{y}{x-x^\gamma}\Big)^\delta-1\right) m(y)dy -\delta\frac{x^\gamma}{x-x^\gamma} \biggr].\]\

\end{lemm}

\begin{proof}
Assume everywhere in the sequel that $x$ is very large. The parameter $\delta$ is allowed to be positive or negative.
\begin{align*}
Dg(x) &= \int_{\BbR^+} [(x-x^\gamma +y)^\delta- x^\delta] m(y)dy\\
\cut{&= (x-x^\gamma)^\delta \int_{\BbR^+} \left[\left(1 +\frac{y}{x-x^\gamma}\right)^\delta- \left(\frac{x}{x-x^\gamma}\right)^\delta\right] m(y)dy\\ }
&= (x-x^\gamma)^\delta \int_{\BbR^+} \left [\left(1 +\frac{y}{x-x^\gamma}\right)^\delta- \left(1+\frac{x^\gamma}{x-x^\gamma}\right)^\delta\right] m(y)dy\\
&\asymp (x-x^\gamma)^\delta \left[\int_{\BbR^+} \left(1 +\frac{y}{x-x^\gamma}\right)^\delta m(y)dy- 1-\delta\frac{x^\gamma}{x-x^\gamma}\right].
\end{align*}
For arbitrary $y_0\in\BbR^+$, the integral $\int_{\BbR^+}$ in the previous formula can be split into $\int_{]0,y_0[} +\int_{[y_0,\infty[}$. In the sequel we shall consider only the case $x\gg y_0$.  If $\delta<0$ then the function 
$y\mapsto (1+ \frac{y}{x-x^\gamma})^\delta$ is decreasing, hence $\sup_{y\in ]0,y_0[} (1+ \frac{y}{x-x^\gamma})^\delta\leq 1$. On the contrary, when $\delta>0$, the corresponding function is increasing and we have
$\sup_{y\in]0, y_0[} (1+ \frac{y}{x-x^\gamma})^\delta\leq (1+ \frac{y_0}{x-x^\gamma})^\delta\asymp 1+\delta \frac{y_0}{x-x^\gamma}$. In any situation, 
\[\int_{]0,y_0[}\Big(1+ \frac{y}{x-x^\gamma}\Big)^\delta m(y)dy\Yleft \mu(]0,y_0[)+ |\delta| \frac{y_0}{x-x^\gamma}.\] 
The remaining integral can be written as
\[\int_{[y_0,\infty[}\Big(1+ \frac{y}{x-x^\gamma}\Big)^\delta m(y)dy=\int_{[y_0,\infty[}\left[\Big(1+ \frac{y}{x-x^\gamma}\Big)^\delta-1\right] m(y)dy+\mu([y_0,\infty[).\]
Replacing these expressions into the formula for $Dg(x)$ yields
\[Dg(x){\color{red}\Yleft} (x-x^\gamma)^\delta \biggl[\int_{[y_0,\infty[}\left(\Big(1+ \frac{y}{x-x^\gamma}\Big)^\delta-1\right) m(y)dy -\delta\frac{x^\gamma}{x-x^\gamma}\biggr],\]
because, for $x$ sufficiently large, $\frac{y_0}{x-x^\gamma}$ is negligible compared to $\frac{x^\gamma}{x-x^\gamma}$.
\end{proof}

\begin{remn}
Note that since $0<\gamma<1$, the asymptotic majorisation
$Dg(x)\Yleft x^\delta \biggl[\int_{[y_0,\infty[}\left(\Big(1+ \frac{y}{x-x^\gamma}\Big)^\delta-1\right) m(y)dy -\delta\frac{x^\gamma}{x-x^\gamma}\biggr]$ is \textit{equivalent} to the one established in lemma \ref{lem:Dg}.
\end{remn}

\begin{lemm}
\label{lem:Dg-B}

Let $\delta<\theta<1$. Suppose further that there exist constants $0<b_1\leq b_2<\infty$ such that for all $y\geq y_0$, for some $y_0>0$, we have 
$b_1\leq c_y\leq b_2$. 
Then,   the integral
\[I(x):= \int_{[y_0,\infty[}\left(\Big(1+ \frac{y}{x-x^\gamma}\Big)^\delta-1\right) m(y)dy,\]
asymptotically for large $x$, satisfies 

 \[\delta B_1K_{\delta,\theta} x^{-\theta}\Yleft I(x)\Yleft \delta B_2K_{\delta,\theta} x^{-\theta},\]
  where $K_{\delta,\theta}=\frac{\Gamma(1-\theta)\Gamma(\theta-\delta)}{\theta \Gamma(1-\delta)}$,
 $(B_1,B_2)=(b_1, b_2)$ if $\delta>0$, and $(B_1, B_2)= (b_2, b_1)$ when $\delta<0$.
\end{lemm}
\begin{proof}
Write
\begin{eqnarray*}
I(x):= \int_{[y_0,\infty[}\left(\Big(1+ \frac{y}{x-x^\gamma}\Big)^\delta-1\right) m(y)dy
= \int_{[y_0,\infty[}c_y\frac{(1+ \frac{y}{x-x^\gamma})^\delta-1}{ y^{1+\theta}}dy.
\end{eqnarray*}
Consider first $\delta>0$; in this case  the integrand is positive, hence 
\[b_1I_1(x)
\leq I(x) \leq
b_2I_1(x),\]
where $I_1(x):=\int_{[y_0,\infty[} \frac{\left(1+ \frac{y}{x-x^\gamma}\right)^\delta-1}{y^{1+\theta}}dy$. 
We estimate then, for fixed $y_0$ and large $x$ (so, $y_0$ is small compared to $x$) and performing the change of variable $u=\frac{y}{x-x^\gamma}$, 
\begin{align*}
I_1(x)&:=   \int_{[y_0,\infty[}\frac{(1+ \frac{y}{x-x^\gamma})^\delta-1}{ y^{1+\theta}}dy  \\
&= (x-x^\gamma)^{-\theta}\int_{\frac{y_0}{x-x^\gamma}} \frac{(1+u)^\delta-1}{u^{1+\theta}} du\Yleft x^{-\theta}\int_0^\infty \frac{(1+u)^\delta-1}{u^{1+\theta}} du 
\end{align*}
Now for $\delta<\theta>1$ (recall that $\theta>0)$ 
\begin{align*}
\int_0^\infty \frac{(1+u)^\delta-1}{u^{1+\theta}} du &=  -\frac{1}{\theta} \int_{u=0}^{u=\infty} [(1+u)^\delta -1]d(u^{-\theta})\\
&= \Big[\frac{1}{\theta}\frac{(1+u)^\delta -1}{u^{\theta}}\Big]_{u=0}^{u=\infty}
+\frac{\delta}{\theta} \int_0^\infty \frac{(1+u)^{\delta-1}}{u^\theta} du\\
&= 0 +\frac{\delta}{\theta}  \frac{\Gamma(1-\theta) \Gamma(\theta-\delta)}{\Gamma(1-\delta)} = \delta K_{\delta,\theta}.
\end{align*}
The claimed majorisation $I_1(x)\Yleft x^{-\theta}\delta K_{\delta,\theta}$ is obtained immediately.
The minoration is obtained similarly. 
If $\delta<0$, the integrand is negative, hence the role of $b_1$ and $b_2$ must be interchanged. 
\end{proof}

\begin{lemm}
\label{lem:Dg-c}
Let $\delta<\theta<1$. Suppose further that $c_y\rightarrow c$. Then for all $\eps >0$, there exists a $y_0$ such that 
\[I(x):= \int_{[y_0,\infty[}\left((1+ \frac{y}{x-x^\gamma})^\delta-1\right) m(y)dy= c\delta K_{\delta, \theta} (x-x^\gamma)^{-\theta}(1+\eps\cO(1)),\]
where $K_{\delta,\theta}=\frac{\Gamma(1-\theta)\Gamma(\theta-\delta)}{\theta\Gamma(1-\delta)}$.
\end{lemm}
\begin{proof}
Observe that $I(x)=cI_1(x) + \int_{[y_0,\infty[}(c_y-c)\frac{(1+ \frac{y}{x-x^\gamma})^\delta-1}{ y^{1+\theta}}dy$.
Now, since $c_y\rightarrow c$, it follows that for all $\eps >0$ one can choose $y_0$ such that for $y\geq y_0$, we have 
$|c_y-c|\leq \eps$. We then immediately conclude that the absolute value of the above integral is majorised by 
$\eps I_1(x)$. 
\end{proof}

\begin{lemm}
\label{lem:Kzero}
\begin{enumerate}
\item Let $\theta\in]0,1[$ and $c>0$. For all $\delta\in]-\infty, \theta[$ let $K_{\delta,\theta}=\frac{1}{\theta}
\frac{\Gamma(1-\theta)\Gamma(\theta-\delta)}{\Gamma(1-\delta)}$.
If $c\pi \csc(\pi \theta) <\theta$ then there exists a unique $\delta_0:=\delta_0(c,\theta)\in ]0, \theta[$ such that 
$cK_{\delta_0,\theta} -1 =0$.
\item
For $\theta\in]0,1[$, $c>0$,  and $\delta\in]-\infty, \theta[$, define $ L_{\delta,\theta}= \frac{\Gamma(1+\delta)\Gamma(-\theta)}{\Gamma(1-\theta-\delta)}$. Then for every fixed $\theta\in]0,1[$, $L_{\delta,\theta}+\frac{1}{\theta}<0$ for all $\delta>0$ and
there exists a unique $\delta_0:=\delta_0(c,\theta)\in ]0,\infty[$ such that $cL_{\delta_0,\theta}+\delta_0=0$.
\end{enumerate}
\end{lemm}

\begin{proof}
\begin{enumerate}
\item
By standard results\footnote{We used the identity $\Gamma(z)\Gamma(1-z)= \pi\csc(\pi z)$, (see \cite[formula \S1.2 (6)]{ErdelyiMagnusOberhettingerTricomi1981a} for instance).}  on $\Gamma$ functions, $K(0,\theta)= \frac{1}{\theta} \csc (\pi \theta)$. For fixed $\theta$, the function $K_{\delta,\theta}$ is strictly increasing  and continuous in $\delta$ --- as follows from its integral representation ---and $\lim_{\delta\uparrow \theta} K_{\delta, \theta} = \infty$, from which follows the existence of $\delta_0\in]0,\theta[$ verifying the claimed equality.  
\item
From \cite[\S1.2, formul\ae\ (4) and (1)]{ErdelyiMagnusOberhettingerTricomi1981a} follow immediately that   $L_{0, \theta}=-\frac{1}{\theta}<0$ and 
$\lim_{\delta\to\infty} \frac{L_{\delta,\theta}}{\Gamma(-\theta) \delta^\theta}=1$. The strict monotonicity and continuity (in $\delta$) of the  function $L_{\delta,\theta}$ follows from its integral representation: $L_{\delta,\theta}+\frac{1}{\theta} =\int_0^1\frac{\left((1-u)^\delta-1\right)}{u^{1+\theta}}du$. 
Hence $L_{\delta,\theta} \asymp \Gamma(-\theta) \delta^\theta$ for large $\delta$. Since the asymptotic behaviour of $L_{\delta,\theta}$ is negative and sublinear in $\delta$, it follows that there exists a sufficiently large $\delta_0$ for which the claimed equality holds. Additionally, the strict monotonicity of $L_{\delta,\theta}$ combined with the fact that $L_{0, \theta}=-\frac{1}{\theta}<0$ guarantees that $L_{\delta,\theta}+\frac{1}{\theta}<0$ for all $\delta>0$.
\end{enumerate}
\end{proof}

\begin{lemm}
\label{lem:Tdynsys}
Let $f:\BbR_+\to \BbR_+$ be a given function;  define the dynamical system $(X_t)_{t\in\BbN}$ by
$X_0=x_0$ and recursively $X_{t+1}=f(X_t)$, for $t\in\BbN$. For $a>0$ define $T_{[0,a]}(x_0)=\inf\{t\geq 1: X_t \leq a\}$.
\begin{enumerate}
\item If $f(x)=x-x^\gamma$ for some $\gamma\in ]0, 1[$ and  $x_0\gg 1$, then $T_{[0,a]}(x_0)\asymp \frac{x_0^{1-\gamma}}{1-\gamma}$. 
 \item If $f(x)=x-x^\gamma+1$ for some $\gamma\in ]0, 1[$,  $a>1$, and $x_0\gg a$, then
 $T_{[0,a]}(x_0)\asymp \frac{x_0^{1-\gamma}}{1-\gamma}$. 
\end{enumerate}
\end{lemm}
\begin{proof}
\begin{enumerate}
\item
The derivative of the function $f$ satisfies $0<f'(x)<1$ for all $x>1$. Therefore, successive iterates 
$f^{\circ n} (x_0)$ eventually reach the interval $[0,1]$ for all $x_0>1$ in a finite number of steps $T_{[0,1]}(x_0)$. To estimate this number, start by approximating, for $X_t=x$ and $1<x<x_0$,  the difference $X_{t+1}-X_t=\Delta X_t =-X_t^\gamma $ by the differential $dX_t =-X_t^\gamma dt$. 
Then
\[T_{[0,a]}(x_0)=\int _0^{T_{[0,a]}(x_0)} dt= -\int_{x_0} ^a X_t^{-\gamma} dX_t=\frac{1}{1-\gamma}(x_0^{1-\gamma}-
a^{1-\gamma})\asymp \frac{x_0^{1-\gamma}}{1-\gamma}.\] 
\item Using the same arguments, and denoting by $F$ the hypergeometric function, we estimate
(see \cite{ErdelyiMagnusOberhettingerTricomi1981a} for instance);
\[T_{[0,a]}(x_0)=\int_{x_0} ^a \frac{dX_t}{1-X^\gamma}= x_0F(1, \frac{1}{\gamma}, 1+\frac{1}{\gamma}, x_0^\gamma)-aF(1, \frac{1}{\gamma}, 1+\frac{1}{\gamma}, a^\gamma)\asymp \frac{1}{1-\gamma} x_0^{1-\gamma}.\]
\end{enumerate}
\end{proof}

\noindent
\textit{Proof of the theorem \ref{thm:precise-right}:}
First we need to prove accessibility of $A=[0,a]$, with $a>1$ from any point $x>a$. Denote by $r:=\mu([0,1])>0$. Since the dynamical system evolving according to the iteration of the function$f(x)=x-x^\gamma+1$ reaches $A$ in finite time $T_A(x)$, as proven in lemma \ref{lem:Tdynsys}, the Markov chain can reach $A$ in time $\tau_A$ verifying 
$\BbP_x(\tau_A\leq T_A+1)\geq Cr^{T_A(x)}>0$, for all $x\gg a$.

We substitute  the estimates obtained in lemmata \ref{lem:Dg-B} and \ref{lem:Dg-c} into the expression for $Dg$ obtained in lemma \ref{lem:Dg}. 
\begin{enumerate}
\item Assume that $b_1\leq c_y\leq b_2$.
\begin{enumerate}
\item Choose  $0<\delta <\theta $. Then 
\begin{eqnarray*}
Dg(x)&=&(x-x^\gamma)^\delta
\left[
-\delta \frac{x^\gamma}{x-x^\gamma}
+ b_2\delta K_{\delta,\theta} (x-x^\gamma)^{-\theta}+ \cO(x^{-1})\right]
\\
&=&  -\delta x^{\delta+\gamma-1}+ \delta b_2 K_{\delta,\theta} x^{\delta-\theta}+ \cO(x^{\delta-\theta-1}).
\end{eqnarray*}
If $\theta>1-\gamma$, the dominant term reads $-\delta x^{\delta+\gamma-1}$ 
which is negative. Hence, $(g(\zeta_n))$ is a supermartingale tending to infinity if $\zeta_n\to \infty$. We conclude then by  theorem \ref{thm:FMM95}.

\begin{itemize}
\item
To prove finiteness  of moments up to $\theta/(1-\gamma)$, consider $p$ such that $0<p\delta< \theta$. Then 
\[Dg^p(x)\Yleft-\delta p x^{\delta p+\gamma-1} =g(x)^{p-\frac{1-\gamma}{\delta}}\leq -C g(x) ^{p-2},\] provided that $\frac{1}{\delta}<\frac{2}{\gamma-1}$. The latter, combined with the inequality $p\delta <\theta$, establishes the majorisation by $-C g(x)^{p-2}$. This allows to conclude by 
 theorem \ref{thm:AIM96}.
\item To prove the non existence of moments for $q\geq\theta/(1-\gamma)$, denote by $f(x)= x-x^\gamma$.    
Define $Z_0=x$ and recursively $Z_{n+1}=f(Z_n)$ as in lemma \ref{lem:Tdynsys};  similarly the Markov chain can be rewritten $\zeta_0=x$ and recursively $\zeta_{n+1}=f(\zeta_n +\alpha_{n+1})$ as long as $\zeta_n>1$.

Now remark that
$Z_1 = f(x) <f(x+\alpha_1)=\zeta_1$; a simple recursion shows that
$Z_{n+1}  = f^{\circ n} (x+\alpha_1) <\zeta_{n+1}$. 
Obviously $T_{[0,1]}(x+\alpha_1, 0) <\tau_0$. Hence $\tau_0>C(x+\alpha_1) ^{1-\gamma}>C(\alpha_1)^{1-\gamma}$ by  lemma \ref{lem:Tdynsys} and subsequently $\BbE_x(\tau_0^q) \Yright C\BbE(\alpha_1) ^{q(1-\gamma)}=\infty$  whenever $q(1-\gamma)\geq \theta$.
\end{itemize}
\item Choose now $\delta <0 $. Using the same arguments as above, we see that 
the dominant term is $\delta b_1K_{\delta,\theta} x^{\delta-\theta}$ which is again negative. Hence $(g(\zeta_n))$ is a bounded supermartingale. We conclude by using    theorem \ref{thm:FMM95}.

\end{enumerate}
\item  Assume now that $\theta=1-\gamma$ and $c_y\rightarrow c>0$. In this situation, for every $\eps >0$ we can choose $y_0$ such that for $y\geq y_0$, we have asymptotically, for $x\gg y_0$ and every $\delta\ne0$,
\[Dg(x)= \delta x^{\delta+\gamma-1}\left(cK_{\delta,\theta}-1 + \cO(x^{-1})+\eps \cO(1)\right).\]
Therefore, the dominant term is $\delta(cK_{\delta,\theta} -1) x^{\delta+\gamma-1}$. The sign of $\delta$ will thus be multiplied by the sign of 
the difference $cK_{\delta,\theta} -1$. 
\begin{enumerate}
\item If $c\pi\csc(\pi\theta)<\theta$,  by lemma \ref{lem:Kzero}, we can chose $\delta\in]0,\delta_0[$, so that that $Dg(x)\leq 0$ while $g$ tends to infinity. We conclude by    theorem \ref{thm:FMM95}. 
\begin{itemize}
\item
To prove finiteness of  moments of the time $\tau_A$, for the  $\delta$ chosen to establish recurrence, we can further choose $p>1$ so that $p\delta<\delta_0$.  Then 
\[Dg^p(x) \Yleft -p\delta x^{p\delta+\gamma-1} =-p\delta g(x)^{p-\frac{1-\gamma}{\delta}}\leq -C g(x)^{p-2}\] whenever  $\frac{1-\gamma}{\delta} >2$ or $\frac{1}{\delta} \leq \frac{2}{1-\gamma}$. Combining with the condition $p\delta<\delta_0$ we get $p<\frac{2\delta_0}{1-\gamma}$ and we conclude by    theorem \ref{thm:AIM96} that all moments up to $\frac{\delta_0}{1-\gamma}$ are finite.
\item
To prove non-existence of moments for $q>\frac{\delta_0}{1-\gamma}$,  for any $\delta\in ]0,\delta_0[$, we  check immediately $Dg(x)\geq -\epsilon$. Now, choose $r>1$ such that $r\delta>\delta_0$ and determine under which circumstances $Dg^r(x)\leq C g(x)^{r-1}$.  Computing explicitly, we get 
\[Dg^r(x)\asymp r\delta (cK_{r\delta,\theta}-1) x^{r\delta +\gamma-1}\leq C g(x)^{r-\frac{1-\gamma}{\delta}}\leq C g(x)^{r-1}\]
 whenever $\frac{1-\gamma}{\delta}> 1$ or equivalently $\frac{1}{\delta} > \frac{1}{1-\gamma}$. But the latter inequalities are always verified for $0<\delta<\delta_0$. Similarly, for any $p$ such that $p\delta>\delta_0$, i.e.\ for $p>\frac{\delta_0}{\delta}>\frac{\delta_0}{1-\gamma}$, we get $Dg^p(x)\geq 0$. We conclude, by theorem \ref{thm:AIM96}, that all moments $q>\frac{\delta_0}{1-\gamma}$ of $\tau_A$ fail to exist.  
\end{itemize} 
\item If $\delta<0$ and $c\pi\csc(\pi\theta)>\theta$, then $(g(\zeta_n))$ is a bounded supermartingale. We conclude  by theorem \ref{thm:FMM95}.

\end{enumerate}
\end{enumerate}
\eproof

\begin{lemm}
\label{lem:Dg-ht-left}
Let $(\zeta_n)$ be the Markov chain of the theorem \ref{thm:precise-left} and assume that $x$ is very large. 
For the Lyapunov function $g$ with $\delta<\theta<1$, we have
\begin{eqnarray*}Dg(x)&= &(x+x^\gamma)^\delta \int_0^{x+x^\gamma} 
\left[\left( 1-\frac{y}{x+x^\gamma}\right)^\delta -1\right]\mu(dy)\\
&&+\left( (x+x^\gamma)^\delta-x^\delta\right) \mu([0,x+x^\gamma])-x^\delta \mu([x+x^\gamma, \infty[).
\end{eqnarray*}
\end{lemm}
\begin{proof}
Write simply
\begin{eqnarray*}
Dg(x)&=& \int_{\BbR^+} \left((x+x^\gamma-y)^+\right)^\delta \mu(dy) -x^\delta\\
&=& (x+x^\gamma)^\delta \int_0^{x+x^\gamma} \left(1-\frac{y}{x+x^\gamma}\right)^\delta \mu(dy) -x^\delta\\
&=&  (x+x^\gamma)^\delta \int_0^{x+x^\gamma} \left[\left(1-\frac{y}{x+x^\gamma}\right)^\delta -1\right]\mu(dy)\\
&&+(x+x^\gamma)^\delta \mu([0, x+x^\gamma]) -x^\delta.
\end{eqnarray*}
\end{proof}

\noindent
\textit{Proof of the theorem \ref{thm:precise-left}:}
First we need to establish accessibility of the state 0. But this is obvious since from any $x>0$ the $\BbP(\alpha_1>x+x^\gamma)>0$. 

We only sketch the proof since it uses the same arguments as the proof of the theorem \ref{thm:precise-right}.
It is enough to consider the case $c_y=c$ since the case $c_y\rightarrow c$ will give rise to an additional corrective term that will be negligible. With this proviso, the integral appearing in the right hand side of the expression for $Dg(x)$ in the previous lemma \ref{lem:Dg-ht-left} reads
\begin{align*}\int_0^{x+x^\gamma} 
\left[\left( 1-\frac{y}{x+x^\gamma}\right)^\delta -1\right]\mu(dy)&=c(x+x^\gamma)^{-\theta} \int_0^1 \frac{\left((1-u)^\delta-1\right) }{u^{1+\theta}}du\\
&= c(x+x^\gamma)^{-\theta} (L_{\delta,\theta}+\frac{1}{\theta}),
\end{align*}
where $ L_{\delta,\theta}$ is defined in lemma \ref{lem:Kzero}..  
It is further worth noting  that $ L_{\delta,\theta}\leq 0$, for all $\delta\in\BbR_+$.
Therefore,
\begin{enumerate}
\item If $\theta<1-\gamma$, then the dominant terms in the expression of $Dg$ are those with $x ^{\delta-\theta}$, hence, choosing $\delta >0$, we get
$Dg(x) \leq  cx^{\delta-\theta} L_{\delta,\theta}$. Since the value of $Dg(x)$  is always negative i.e.\ the process $(g(\zeta_n))$ is a supermartingale tending to infinity. We conclude  by theorem \ref{thm:FMM95}.

To establish the existence of all moments, it is enough to check that 
\[Dg^p(x) \Yleft 
cx ^{p\delta-\theta} L_{p\delta,\theta}\Yleft -Cg(x) ^{p -\frac{\theta}{\delta}}\leq -Cg(x) ^{p -2}\] whenever $\delta>\theta/2$. But since $L_{\delta,\theta}$ is defined and negative for all positive $\delta$, we conclude that all
positive moments of $\tau_0$ exist by theorem \ref{thm:AIM96}.
\item
When $\theta=1-\gamma$, then all terms are of the same order and 
$Dg(x) \Yleft x^{\delta-\theta} (cL_{\delta,\theta}+\delta)$.
From lemma \ref{lem:Kzero}, for fixed $\theta$ and $c>0$,  there exists $\delta_0>0$ such that 
$cL_{\delta_0,\theta} +\delta_0=0$. We conclude then that asymptotically, for large $x$, 
\[Dg(x)\Yleft x ^{\delta-\theta} (cL_{\delta, \theta}-1),\]
the sign of the discrete Laplacian is negative (positive) depending on the value of $\delta$ being smaller (larger) than $\delta_0$.  

Choose $\delta>0$ and $p$ such that $p\delta<\delta_0$. Then $Dg^p(x)\Yleft -C g(x)^{p-\frac{\theta}{\delta}}\leq -C g(x)^{p-2}$ whenever $\frac{1}{\delta}<\frac{2}{\theta}$ and, consequently, $p<\frac{2 \delta_0}{\theta}$. Then we conclude by theorem \ref{thm:AIM96}
that $\BbE_x (\tau_0^q)<\infty$ for all $q< \frac{ \delta_0}{\theta}$ as claimed.

To show that moments higher than $ \frac{ \delta_0}{\theta}$ fail to exist, choose $\delta <\delta_0$. It is  then evident that $Dg(x)\asymp -C x^{\delta-\theta}\geq -\epsilon$, for some $\epsilon>0$.
There exists  then $r>1$ such that $r\delta>\delta_0$; estimating then  $Dg^r(x)\asymp Cg(x)^{r-\frac{\theta}{\delta}}$ we conclude immediately that $0\leq Dg^r(x)\leq C g(x)^{r-1}$ whenever $\frac{\theta}{\delta}>1$. We conclude then by theorem \ref{thm:AIM96}
 that for all $q>\frac{\delta_0}{\theta}$, we have 
$\BbE_x(\tau_0^q)=\infty$.
\item If $\theta>1-\gamma$, the dominant term is $\delta x^{\delta-\gamma+1} \mu([0, x+x^\gamma])$ that can be made negative by choosing $\delta <0$ and $x$ sufficiently large. We conclude by
theorem \ref{thm:FMM95}.
\end{enumerate}
\eproof

\begin{remn}
\label{rem:sum-integral} 
In this subsection, we assumed that the law $\mu$ of the random variables $(\alpha_n)$ is absolutely continuous with respect to the Lebesgue measure on $\BbR^+$. If instead the law is absolutely continuous with respect to the
counting measure on the positive integers, the integrals in the expression of $Dg$ become sums. Now, the sums over the positive integers can be replaced by integrals. It turns out that the error committed in such a replacement is \textit{always}  a subleading term in the expression of $Dg$, leaving the conclusion unaffected. 
\end{remn}

\begin{remn}
\label{rem:both-sided} 
The two previous theorems have been established by assuming that the random variables $(\alpha_n)$ are always positive and act in the opposite direction of the systematic drift $x^\gamma$. By examining the proofs of the theorems however, it is evident that nothing will change if the random variables are both sided, even with both sided heavy tails, provided that the heaviest tail is the one acting in the opposite direction of the systematic drift $x^\gamma$.
\end{remn}

\subsection{Proof of the theorems \ref{thm:systematic-left} and  \ref{thm:systematic-right}}
Here the control is only through the tail decay and consequently, the estimates are considerably more involved. 
The subsection relies on methods developed in 
\cite{HrynivMacPheeMenshikovWade2012} to deal with heavy tails when only tail control is available.

\begin{lemm}
\label{lem:inequality}
Let $Z$ be a positive random variable, $\phi:\BbR^+\rightarrow \BbR^+$ an increasing 
function,  and $0\leq a < b\leq \infty$. Then 
\[\BbE(\phi(Z)\id_{[a,b[} (Z)) = \int_{[\phi(a),
\phi(b)[} \BbP(Z>\phi^{-1}(t)) dt- \phi(b)\BbP(Z\geq b)+\phi(a) \BbP(Z \geq a).\]
\end{lemm}
\begin{proof}
Denote by $\nu$ the law of $Z$. Then
\begin{eqnarray*}
\int_{[a,b[} \BbP(Z>t) dt&=& \int_{[a,b[} \BbE(\id_{]t,\infty[} (Z)) dt \\
&=&\int_{\BbR^+\times \BbR^+} \id_{]t,\infty[}(z) \id_{[a,b[} (t)\  dt\ \nu(dz)\\
&=&\int_{\BbR^+\times \BbR^+} \id_{[a,z\wedge (b-)]}(t) \  dt\ \nu(dz)\\
&=& \int_{[a,\infty[} [z\wedge (b-)-a]  \nu(dz) \\
&=& \BbE\left(Z \id_{[a,b[}(Z)\right) + b \BbP(Z\geq b)-a\BbP(Z\geq a).
\end{eqnarray*}
On denoting $Y=\phi(Z)$, we conclude by remarking 
that 
\[\BbE\left(\phi(Z) \id_{[a,b[}(Z)\right)=
\BbE\left(Y\id_{[\phi(a), \phi(b)[}(Y)\right).\]
\end{proof}
\begin{remn}{}
When $b=\infty$ in the above formula and the random variable $Z$ is almost surely finite, then the  term $b\BbP(Z\geq b)$ reads $\infty\ \BbP(Z=\infty)=0$; otherwise the value is $\infty$ and the random variable $Z$ cannot be then almost surely finite.
\end{remn}

In the sequel, we shall partition the real axis into $\BbR=\sqcup _{i=1}^4 A_i$ with  
\[ A_1=]-\infty, -x^\beta[, \quad A_2=[-x^\beta, 0[,\quad  A_3=[0,x^\beta[,\quad  A_4=[x^\beta, \infty[,
 \]
 with some parameter $\beta$ (verifying $0<\gamma<\beta<1$) that will be specified later.
 On denoting, 
 for every choice of the Lyapunov function $g$, by $d_i=\BbE (g(\zeta_{n+1})\id_{A_i}(\alpha_{n+1})|\zeta_n=x)$,  the above partition induces a decomposition of 
 the conditional increment as 
 \[
  Dg(x)=\sum_{i=1}^4 (d_i-g(x)\mu(A_i)).
 \]
 
\vskip4mm 
\noindent
\textit{Proof of the theorem \ref{thm:systematic-right}}.

\begin{enumerate}
\item Let  $\beta\in ]\gamma,1[$ and  $\delta>0$ and define 
\[
 g(x)=\left\{\begin{array}{rl}
              x^{-\delta},&x\ge 1,\\
                1,&x<1.
             \end{array}
\right.\]
  The parameter $\delta$ (together with $\beta$) will be chosen later; we get then
  \[d_i=\int_{A_i} g((x+x^\gamma+y)^+) \mu(dy).\] 
For $x$ sufficiently large we have
\begin{eqnarray*}
d_1&\leq& \mu(A_1),\\
d_2 &=&x^{-\delta} \int_{A_2} \left(1+\frac{x^\gamma+y}{x}\right)^{-\delta}\mu(dy)\\
&\asymp& x^{-\delta} (1-\delta x^{\gamma-1}) \mu(A_2) -\delta x^{-\delta-1} \int_{A_2} y \mu(dy)\\
& \asymp&  x^{-\delta} \mu(A_2)  -\delta x^{-\delta+\gamma-1}\mu(A_2) +\delta x^{-\delta-1} \int_{A_2}| y| \mu(dy),\\
d_3 &\asymp& x^{- \delta} \mu(A_3) -\delta x^{-\delta+\gamma-1}\mu(A_3)-\delta x^{-\delta-1} \int_{A_3} y \mu(dy),\\
d_4&\Yleft &(x+x^\gamma+x^\beta)^{-\delta} \mu(A_4)\\
& \asymp& x^{-\delta} \mu(A_4) -\delta x^{-\delta+\gamma-1} \mu(A_4)-\delta x^{-\delta+\beta-1} \mu(A_4).
\end{eqnarray*}
Replacing into the expression for $Dg$, we get

\begin{align*}
Dg(x) &\Yleft  \mu(A_1) + \delta  x^{-\delta-1} \int_{A_2}| y| \mu(dy)\\
&\mbox{  } -\delta x^{-\delta+\gamma-1}[\mu(A_2)+\mu(A_3)+\mu(A_4)]  -\delta x^{-\delta-1} \int_{A_3} y \mu(dy).
\end{align*}
Note that in the previous inequality, the terms on the first line are positive, while the terms appearing in the second line are negative.  In order that $Dg$ be negative, we need to show that the positive terms are subdominant in the expression of $Dg(x)$ for sufficiently large $x$.  Now, $\mu(A_1)=\BbP(\alpha_1<-x^\beta) \leq Cx^{-\beta\theta}$, while, by lemma \ref{lem:inequality},
\[ \int_{A_2}| y| \mu(dy) =\BbE(\alpha_1^-\id_{]0,x^\beta]}(\alpha_1^-)\leq
\int_0^{x^\beta} \BbP(\alpha_1^->t)dt-x^\beta \BbP(\alpha_1^->x^\beta)\leq x^\beta- C x^{\beta(1-\theta)}\Yleft x^\beta.\]

Combining, we see that we get a supermartingale if we satisfy simultaneously the inequalities

\[
 -\theta \beta<\gamma -\delta -1  \ \textrm{and} \  -\delta-1+\beta <\gamma-\delta-1
\]
that --- for $\theta>1-\gamma$ --- have a solution for $\beta\in]\frac{1-\gamma}{\theta}, 1[$  and $\delta\in ]0, \beta\theta-(1-\gamma)[$. We conclude by 
theorem \ref{thm:AIM96}.

\item
Let $\beta>0$, $\delta\in ]0, \theta[$, and    $g(x)=x^\delta$.  
The possible values of the parameters
 $\beta$ and $\delta$ will be further delimited later. We proceed now with the partition
 $\BbR=\sqcup_{i=1}^4 A_i$, where
 $A_1=]-\infty,-x^{\beta}[$,  $A_2=[-x^{\beta}, 0[$,
 $A_3=[0, x^\beta[$, and $A_4=[x^\beta, \infty[$; we introduce also the sets $A_0=[-x-x^\gamma, -x^\beta[\subset A_1$ and $B=A_1\setminus A_0=]-\infty, -x-x^\gamma[$.
Using similar arguments as  in the first part of the theorem
we estimate
\begin{align*}
d_1&=\int_{A_1} \left((x+x^\gamma + y)^+\right)^\delta \mu(dy)= \int_{A_0} (x+x^\gamma+y)^\delta \mu(dy)\\
&\leq (x+x^\gamma-x^\beta)^\delta \mu(A_0)\asymp x^\delta\mu(A_0) +\delta x^{\delta+\gamma-1} \mu(A_0)-\delta x^{\delta+\beta-1}\mu(A_0), 
\end{align*}
leading further to the estimate 
\[d_1-x^\delta\mu(A_1)\Yleft -x^\delta \mu(B) +\delta x^{\delta+\gamma-1} \mu(A_0)-\delta x^{\delta+\beta-1}\mu(A_0).\]
The estimates of the other terms are obtained using the similar arguments:
 \begin{align*}
 d_2-x^\delta\mu(A_2)&\asymp -\delta x^{\delta-1} \BbE\left(|\alpha_1| \id_{A_2}(\alpha_1)\right) +\delta x^{\delta-1+\gamma} \mu(A_2) \leq \delta x^{\delta-1+\gamma} \mu(A_2),\\
 d_3-x^\delta\mu(A_3)&\asymp \delta x^{\delta-1} \BbE\left(\alpha_1 \id_{A_3}(\alpha_1)\right) +\delta x^{\delta-1+\gamma} \mu(A_3),\\
d_4-x^\delta\mu(A_4)&\asymp  \BbE\left(\alpha_1^\delta \id_{A_4}(\alpha)\right)+ x^{\gamma\delta} \mu(A_4) \leq   \BbE\left(\alpha_1^\delta \id_{A_4}(\alpha)\right)+ C'x^{\gamma\delta-\beta\theta'},
 \end{align*}
 where, we have used \cite[\S2.10, p.\ 28]{HardyLittlewoodPolya1988} to establish the inequality $(a+b+c) ^\delta \leq a^\delta+ b^\delta+c^\delta$ that has been used to obtain the estimate for  $d_4$.
 Using lemma \ref{lem:inequality}, we get
\begin{align*}
\BbE\left(\alpha_1 \id_{A_3}(\alpha_1)\right)
&= \BbE\left(\alpha_1^+ \id_{[0,x^\beta]}(\alpha_1^+)\right)=\left(\int_0^{x^{\beta}} \BbP(\alpha_1^+>t) dt - x^{\beta} \BbP(\alpha_1^+ >x^{\beta})\right)\\
&\leq  C' x^\beta (1-\mu(A_4)), \\
\BbE\left(\alpha_1^\delta \id_{A_4}(\alpha_1)\right)
&= \BbE\left((\alpha_1^+)^\delta \id_{[x^\beta, \infty]}(\alpha_1^+)\right)
= \int_{x^{\beta\delta}}^\infty \BbP(\alpha^+>t^{1/\delta}) dt + x^{\beta\delta} \BbP(\alpha^+ >x^\beta)\\
&\leq   C'\int_{x^{\beta\delta}}^\infty t^{-\theta'/\delta} dt + C'x^{\beta(\delta-\theta')} \ \asymp \ 
K x^{\beta(\delta-\theta')},
\end{align*}
 where  $K= C'\frac{\theta'}{{\theta'}-{\delta}}$. (Mind that $\delta<\theta<\theta'$).
 Using the fact that $\mu(B)\Yleft C x^{-\theta}$ and $\mu(A_4)\leq C' x^{-\beta\theta'}$ and 
 grouping the terms together, we get 
\[Dg(x)\Yleft -Cx^{\delta-\theta} +\delta x^{\delta-1+\gamma} + \delta x^{\delta-1+\beta}+ C'x^{\gamma\delta-\beta\theta'} +K x^{\beta(\delta-\theta')}.\]

This conditional increment will be negative for sufficiently large $x$, provided that the following inequalities
\begin{align*}
\delta -\theta > \delta -1 +\gamma &\Leftrightarrow \theta<1-\gamma\\
\delta -\theta >  \delta -1 +\beta &\Leftrightarrow \beta<1-\theta\\
\delta -\theta >\gamma \delta -\beta \theta' &\Leftrightarrow \frac{\theta-\beta \theta'}{1-\gamma} <\delta\\
\delta -\theta > \beta(\delta-\theta')&\Leftrightarrow \frac{\theta-\beta \theta'}{1-\beta} <\delta
\end{align*}
have a non-empty set of solutions.
 Now, the first inequality is automatically verified by the hypothesis of the theorem. Recalling that $\delta<\theta$, the inequalities $\frac{\theta-\beta \theta'}{1-\gamma} <\delta<\theta$ have a non-empty set of solutions for $\delta$ provided that $\beta \in I:=]\gamma\frac{\theta}{\theta'}, 1-\theta[$; but $I\ne \emptyset$, hence such  $\delta$'s exist. Finally, the inequalities $\frac{\theta-\beta \theta'}{1-\beta} <\delta<\theta$ have automatically a non-empty set of solutions since $\theta'>\theta$. Therefore,
 $\forall \beta\in J:= ]\gamma \frac{\theta}{\theta'}, 1-\theta[$, we can choose $\delta\in ]b, \theta[$ --- where $b:=\max(\frac{\theta-\beta \theta'}{1-\gamma}, \frac{\theta-\beta \theta'}{1-\beta})$ so that $Dg(x) \Yleft -Cx^{\delta-\theta}$. 
 
 To establish the existence of moments, choose $p>0$ such that $g^p \in\Dom_+(P)$, i.e.\ $\delta p<\theta$. From the previous statements, we can choose $\beta\in J$ for  $b$ to be arbitrarily close to $0$. Now $Dg^p(x)\Yleft -C x^{\delta p-\theta}\leq -C g(x)^{p-2}$ provided that $p -\frac{\theta}{\delta} >p-2$ or equivalently $\frac{1}{\delta} <\frac{2}{\theta}$. From the condition $p\delta <\theta$ we get $p<2$ hence, by theorem \ref{thm:AIM96}, $\BbE_x(\tau_A^q)<\infty$ for all $q<1$.
 
 \end{enumerate}
  \eproof


\noindent
\textit{Proof of the theorem \ref{thm:systematic-left}.}
Accessibility of $A$ follows using the same arguments as those used in the proof of theorem
\ref{thm:precise-left}.  
We use again the partition $\BbR=\sqcup_{i=1}^4 A_i$, with $A_1=]-\infty,-x^\beta[$, $A_2=[-x^\beta, 0[$, $A_3=[0,x^\beta[$, and $A_4=]x^\beta,\infty[$, with provisional choice of  the parameter $\beta\in]0,1[$; its domain of variation will be further delimited later. For appropriately chosen $g$, we decompose the conditional drift $Dg(x)=\sum_{i=1}^4 (d_i -g(x) )\mu(A_i)$, where $d_i=\int_{A_i} g((x-x^\gamma+y)^+) \mu(dy)$.
\begin{enumerate}
\item Let $g(x)=x^\delta$, with $\delta\in]0,\theta[$ (the domain of $\delta$ will be further delimited later). We get 
\begin{align*}
d_1&\leq (x-x^\gamma-x^\beta) ^\delta \mu(A_1)
\asymp x^\delta\mu(A_1)-\delta x^{\delta-1+\gamma} \mu(A_1) -\delta x^{\delta-1+\beta} \mu(A_1),\\
d_i&\asymp x^\delta \mu(A_i)-\delta x^{\delta-1+\gamma} \mu(A_i) +\delta x^{\delta-1} \int_{A_i} y \mu(dy), \ \textrm{for } i=2,3,\\
d_4&\leq (x-x^\gamma)^\delta +\int_{A_4} y^\delta \mu(dy)\\
 &\asymp x^\delta \mu(A_4) -\delta x^{\delta-1+\gamma} \mu(A_4) + \left[ \int_{x^{\beta\delta}}^\infty
 \BbP(\alpha_1^+>t^{1/\delta}) dt + x^{\beta\delta} \BbP(\alpha_1^+>x^\beta)\right]\\
 &\leq x^\delta \mu(A_4) -\delta x^{\delta-1+\gamma} \mu(A_4)+ K x^{\beta\delta-\beta\theta}, \textrm{ where } K=C(1+\frac{\delta}{\theta-\delta}).
\end{align*}
Now 
\begin{align*}
\int_{A_2} y\mu(dy)&\leq 0\\
\int_{A_3} y\mu(dy)&=\int_0^{x^\beta} \BbP(\alpha_1^+>t) dt -x^{\beta}\mu(A_4)
\leq  x^{\beta}(1 -\mu(A_4)),
\end{align*}
so that, grouping all terms together, we get
\[Dg(x) \Yleft -\delta x^{\delta-1+\gamma} +C\delta x^{\delta-1 +\beta }+ Kx^{\beta\delta-\beta\theta}.\]
This conditional increment will lead to a supermartingale tending to infinity whenever the system of inequalities
\begin{align*}
\delta-1+\gamma > \delta -1 +\beta  &\Leftrightarrow \beta<\gamma, \ \textrm{and}\\
\delta-1+\gamma > \beta\delta -\beta\theta &\Leftrightarrow \delta>\frac{1-\gamma-\beta\theta}{1-\theta}\\
\end{align*}

have a non-empty set of solutions. Recalling that $\delta <\theta$, the second inequality defines a non-empty domain for $\delta$ provided that  $\frac{1-\gamma-\beta\theta}{1-\theta}<\theta\Leftrightarrow \theta>1-\gamma$ which is satisfied by hypothesis. Hence, picking any $\beta\in]0,\gamma[$ and 
$\delta\in  J:=]\frac{1-\gamma-\beta\theta}{1-\theta}, \theta[$ guarantees that 
$Dg(x)\Yleft -\delta x^{\gamma-1-\delta}$ and shows that $(g(\zeta_n))_n$ is a positive supermartingale, while $g\rightarrow \infty$. 
We conclude from theorem \ref{thm:FMM95} that the chain is recurrent. 

To establish the existence of moments, pick again any $\beta\in ]0,\gamma[$ and $\delta, p>0$ such that $\delta p\in J$. Then, by the previous results, 
\[Dg^{\delta p} (x)\Yleft -\delta p x^{\gamma-1+\delta p} = -\delta p g(x) ^{p-\frac{1-\gamma}{\delta}}\leq -\delta p g(x)^{p-2}\] whenever $\frac{1}{\delta} <\frac {2}{1-\gamma}$. We conclude by the theorem \ref{thm:AIM96} that the moments $\BbE_x (\tau_A^q)<\infty$, $\forall q<\frac{\theta}{1-\gamma}$. 
Since $\theta>1-\gamma$, this result establishes in particular that the passage time is integrable.

\item  Let  now $g(x)= x^{-\delta}\id_{[1,\infty[}(x)+\id_{[0,1[}(x)$, with $\delta>0$ and  choose $\beta\in]0,1[$ (the domains of $\delta$ and $\theta$ will be further delimited later). Estimate then  
\begin{align*}
d_1 &\leq \mu(A_1),\\
d_2 & = \int_{A_2} (x-x^\gamma +y)^{-\delta}\mu(dy)\asymp x^{-\delta} \mu(A_2)+\delta x^{-\delta-1+\gamma}\mu(A_2)+\delta x^{\delta-1} \int _{A_2} |y| \mu(dy),\\
d_3 &=\int_{A_3} (x-x^\gamma +y)^{-\delta}\mu(dy)\asymp x^{-\delta} \mu(A_3)+\delta x^{-\delta-1+\gamma}\mu(A_3)-\delta x^{-\delta-1}\int_{A_3} y\mu(dy),\\
d_4& \leq (x-x^\gamma +x^\beta)^{-\delta} \mu(A_4)
\asymp  x^{-\delta} \mu(A_4)+ \delta x^{-\delta-1+\gamma}\mu(A_4)-\delta x^{-\delta-1+\beta}\mu(A_4).\\
\end{align*}
The integrals appearing in the above majorisations can be further estimated --- using lemma \ref{lem:inequality} --- as
\begin{align*}
\int _{A_2} |y| \mu(dy)&= \int_0^{x^\beta} \BbP(\alpha_1^- >t) dt + x^\beta \BbP(\alpha_1^->x^\beta)\leq x^\beta (1+\mu(A_1))\\
\int _{A_3} y \mu(dy)&= \int_0^{x^\beta} \BbP(\alpha_1^+ >t) dt + x^\beta \BbP(\alpha_1^+>x^\beta)\geq x^\beta (1+\mu(A_4)).
\end{align*}

Grouping all terms together, we obtain
\[Dg(x)\leq C'x^{-\beta \theta'} +\delta x^{-\delta-1+\gamma} + C' \delta x^{-\delta-1+\beta(1-\theta')}- 2C \delta x^{-\delta-1+\beta(1-\theta)}.\]

Only the last term in the above expression is negative. For the image of $(\zeta_n)$ through $g$ to be a supermartingale, we must choose the parameters $\beta$ and $\delta$ so that $Dg(x)\leq 0$ for $x$ large enough. The set of solutions to the following inequalities
\begin{align*}
-\delta-1+\beta(1-\theta) > -\delta-1+\beta(1 -\theta') & \Leftrightarrow \theta<\theta', \\
-\delta-1+\beta(1-\theta) > -\delta-1 +\gamma & \Leftrightarrow \beta> \frac{\gamma}{1-\theta}, \\
-\delta-1+\beta(1-\theta) > -\beta \theta' & \Leftrightarrow \delta < \beta(1-(\theta'-\theta))-1\\
\end{align*}
have a non-empty set of solutions. In fact, the first inequality is satisfied by hypothesis; the second imposes reducing the initial domain of $\beta$ to $\beta >\frac{\gamma}{1-\theta}$. Since $\delta$ must be strictly positive, the last inequality defines a non-empty domain for $\delta$ provided that $\beta >\frac{1}{1-(\theta'-\theta)}$. Hence, picking any $\beta\in ]b,1[$ with $b=\max\{\gamma/(1-\theta), 1/[1-(\theta'-\theta)]\}$ and any $\delta\in ]0, \beta(1+\theta-\theta')-1[$ guarantees that 
$(g(\zeta_n)_n$ is a bounded positive supermartingale. We conclude from   theorem \ref{thm:FMM95}.
 \end{enumerate}
 
 \eproof
 
 \section{Conclusion and open problems}
 We have examined the asymptotic behaviour of the chains $(\zeta_n)$ evolving on $\BbR^+$. The cases we reported in this paper demonstrate an interesting phenomenon of antagonism between the heaviness of the tail (quantified by $\theta$) of the innovation part of the Markov chain and the strength of the systematic drift (quantified by $\gamma$). It is precisely this antagonism that makes the model non trivial; if instead of heavy-tailed random variables, integrable ones are used, then the systematic drift totally determines the asymptotic behaviour of $(\zeta_n)$.

 Note also that the study of the chain $(\zeta_n)$ is sufficient for determining whether the limiting behaviour of the original random dynamical system $(X_n)$ is towards 0 or $\infty$. Nevertheless, the Markov chain obtained by looking at the $(X_n)$ on logarithmic scales is not $(\zeta_n)$ (evolving on $\BbR^+$) but $(\xi_n)$ (evolving on $\BbR$). Interesting problems concern
 random dynamical systems in higher dimension   driven by non-integrable random matrices.  

\section*{Acknowledgements} 
  The authors  acknowledge support from various sources:  Vladimir Belitsky from CNPq (grant 307076/2009-1) and 
Fapesp (grants 2009/52379-8 and 2011/51509-5); Mikhail Menshikov from  Fapesp (grant 2011/07000-0); 
Dimitri Petritis from  \textit{Réseau mathématique France-Brésil}; 
Marina Vachkovskaia from CNPq (grant 302593/2013-6), Fapesp (grant 2009/52379-8), and Faepex (grant 1039/09). 

\bibliographystyle{plain}
\scriptsize{
\bibliography{complete-bibliography}
}
\end{document}